\documentclass[11pt]{amsart}

\usepackage[a4paper,margin=1in]{geometry}
\usepackage{amsmath,amssymb,amsthm,mathtools}
\usepackage{booktabs}
\usepackage{enumitem}
\usepackage{tikz-cd}
\usepackage{quiver}
\usepackage{graphicx}
\usepackage{float}
\usepackage{hyperref}
\usepackage[nameinlink,capitalise]{cleveref}
\hypersetup{
  colorlinks=true,
  linkcolor=blue,
  citecolor=blue,
  urlcolor=blue
}

\DeclareMathOperator{\Gal}{Gal}

\DeclareMathOperator{\Div}{Div}
\DeclareMathOperator{\Cl}{Cl}

\DeclareMathOperator{\Ker}{Ker}

\newcommand{\F}{\mathbb F}
\newcommand{\Z}{\mathbb Z}
\newcommand{\C}{\mathbb C}

\newcommand{\Nhat}{\widehat N}

\newcommand{\frakp}{\mathfrak p}
\newcommand{\frakP}{\mathfrak P}
\newcommand{\ac}{\mathrm{ac}}

\newcommand{\defeq}{:=}

\theoremstyle{plain}
\newtheorem{theorem}{Theorem}[section]
\newtheorem{proposition}[theorem]{Proposition}
\newtheorem{lemma}[theorem]{Lemma}

\theoremstyle{definition}
\newtheorem{definition}[theorem]{Definition}

\theoremstyle{remark}

\title[Multiplicator freeness for restricted-ramification Galois groups]{Multiplicator Freeness for Restricted-Ramification Galois Groups over Global Function Fields}
\author{Qi Liu\(^{*}\)}
\address{School of Mathematical Sciences, Nanjing Normal University, Nanjing 210023, China}
\email{QiLiu67@aliyun.com}
\author{Zugan Xing\(^{*}\)}
\address{School of Mathematical Sciences, Shanghai Jiao Tong University\\Shanghai 200240, China}
\email{xingzugan@aliyun.com}
\thanks{\(^{*}\)These authors contributed equally to this work.}

\begin{document}

\begin{abstract}
	Let \(K\) be a global function field with full constant field \(\F_q\) and let \(T\) be a finite set of finite places. For \(S=\{\infty\}\), we prove that \(\operatorname{Gal}(K(\ell,T,S)/K)\) is multiplicator-free if \(\ell\nmid q-1\) or \(T\ne\varnothing\); in particular, if \(T\ne\varnothing\), this holds for every prime \(\ell\). When \(\ell=p=\operatorname{char}K\) and \(T\ne\varnothing\), the same conclusion holds for every non-empty finite splitting set \(S\) disjoint from \(T\). In both cases, the Schur multiplier at every finite Galois level admits a central realization inside \(K(\ell,T,S)\).
\end{abstract}

\maketitle

\section{Introduction}

Let \(K\) be a global function field with full constant field
\(\mathbb F_q\). Fix a rational prime \(\ell\), a finite set \(T\) of finite places of \(K\), and a finite set \(S\) of places disjoint from
\(T\). Let \(K(\ell,T,S)\) denote the maximal pro-\(\ell\) extension of \(K\) which is unramified outside \(T\) and in which every place of \(S\) splits completely. Let
\[
\Omega_K(\ell,T,S)
:=
\operatorname{Gal}\bigl(K(\ell,T,S)/K\bigr).
\]

Class field theory determines the abelianization of
\(\Omega_K(\ell,T,S)\). Beyond this abelian description, the study of its non-abelian structure
naturally leads to group-theoretic invariants, among which the
Schur multiplier
\[
M\bigl(\Omega_K(\ell,T,S)\bigr)
:=
H_2\bigl(\Omega_K(\ell,T,S),\mathbb Z\bigr).
\]
Opolka showed that the triviality of the Schur multiplier of a
profinite group can be characterized in terms of the solvability of
suitable central embedding and lifting problems
\cite{Opolka_1993,Opolka_2010}. 
Central embedding problems play a fundamental role in the study of non-abelian Galois extensions. More general finite embedding problems over global fields, including problems with prescribed local behavior and controlled ramification; see, for example, \cite{Jarden_Ramiharimanana_2019}.
For number fields, Schur multipliers of restricted-ramification Galois
groups were studied by Watt and Ullom. Watt proved multiplicator
freeness for certain maximal pro-\(\ell\) extensions of imaginary
quadratic fields with restricted ramification \cite{Watt_1985}.
Subsequently, Ullom and Watt used multiplicator freeness to determine
generators and relations for certain class-two Galois groups
\cite{Ullom_Watt_1986}. 

Triviality of the Schur multiplier has also
played a role in the study of presentations of
restricted-ramification pro-\(p\) Galois groups
\cite{Boston_Perry_2000,Liu_Xing_2026}. In a different direction, the
non-abelian Cohen--Lenstra heuristics of Boston--Bush--Hajir show that multiplicator freeness
places strong restrictions on the distribution of unramified
\(p\)-class field tower groups \cite{Boston_Bush_Hajir_2017}.

For global function fields, Miyake showed that the Schur multiplier of a finite Galois group can be
realized through suitable central extensions and related this
construction to the Hasse norm obstruction
\cite{Miyake_1983}. Bae and Jung further developed this connection in
their study of the Hasse norm principle over global function fields
\cite{Bae_Jung_2001}. Thus, the Schur multiplier lies naturally at the
intersection of three themes: the structure of restricted-ramification
Galois groups, central embedding problems, and the Hasse norm
principle.

The purpose of the present paper is to prove the triviality of the
Schur multiplier for a natural class of restricted-ramification
pro-\(\ell\) Galois groups over global function fields. Our main result
is the following.

\begin{theorem}\label{thm:intro-main}
Let \(K\) be a global function field with full constant field
\(\mathbb F_q\), and let \(T\) be a finite set of finite places of \(K\).
Then the following hold.
\begin{enumerate}
\item If \(S=\{\infty\}\) and either \(\ell\nmid(q-1)\) or \(T\neq\varnothing\), then
\(\operatorname{Gal}(K(\ell,T,S)/K)\) is multiplicator free.
\item If \(p=\operatorname{char}K\), \(\ell=p\), and \(T\neq\varnothing\), then
\(\operatorname{Gal}(K(p,T,S)/K)\) is multiplicator free for every non-empty finite set \(S\) disjoint from \(T\).
\end{enumerate}
In both cases, for every finite Galois subextension
\[
L/K\subseteq K(\ell,T,S)/K,
\]
the Schur multiplier of \(\operatorname{Gal}(L/K)\) admits a central
realization inside \(K(\ell,T,S)\).
\end{theorem}

The proof is based on a character-theoretic realization of the Schur
multiplier at finite levels of the tower. More precisely, let \(L/K\)
be a finite Galois subextension of \(K(\ell,T,S)/K\), and put
\(\Gamma=\operatorname{Gal}(L/K)\). We prove a factorization of the
form
\[
X_\ell(C_L)^\Gamma
=
X_\ell(C_L)_{T,S}^\Gamma
\cdot
\widehat N_{L/K}(Y),
\]
where \(X_\ell(C_L)_{T,S}^\Gamma\) consists of the
\(\Gamma\)-invariant \(\ell\)-primary idele class characters satisfying
the prescribed ramification and splitting conditions, and \(Y\) is a
suitable group of \(\ell\)-primary characters of \(C_K\). Class field
theory then yields a central realization inside \(K(\ell,T,S)\).

The theorem provides a natural starting point for further questions
concerning central embedding problems, Hasse norm obstructions, and
the structure of restricted-ramification Galois groups over
global function fields. In future studies, we will relate multiplicator freeness to these questions and develop some of its arithmetic consequences.

\section{Preliminaries}\label{sec:preliminaries}

In this section, we introduce the notation and recall the necessary facts from class field theory and the theory of Schur multiplicators.
\subsection{Schur multiplicators and multiplicator-free}

\begin{definition}
	Let \(G\) be a profinite group. We denote its Schur multiplicator (or Schur multiplier), \( H _{2}(G,\mathbb{Z})\), by \( M(G) \). We say that \(G\) is multiplicator free if \( M(G)=1\).
\end{definition}
The following characterization of multiplicator-free groups is well known (see \cite[Proposition 4.1]{Frohlich_1983}): The profinite group \(G\) is multiplicator-free if and only if for every open normal subgroup \(H\) of \(G\), the canonical surjection
\[
	M(G/H) \twoheadrightarrow \frac{H \cap [G,G]}{[H,G]}
\]
is an isomorphism, where \( [H,G]\) is the closed subgroup generated by all commutators \( [h,g]=h^{-1}g^{-1}hg\) with \( h \in H\) and \( g \in G\).

In this paper, we focus on the cases where \(G\) is a Galois group. For a Galois tower \(K \subset L \subset F \subset \overline{K}\) where \( \overline{K}\) is the algebraic closure of \(K\), let \( \Omega = \operatorname{Gal}(\overline{K} /K)\), \( G _{F} = \operatorname{Gal}(\overline{K} /F)\), \( G _{L} = \operatorname{Gal}(\overline{K} /L)\), and \( G = \operatorname{Gal}(L /K)\). By the basic property of Schur multiplier \cite[Proposition 3.1]{Frohlich_1983}, there exists two canonical surjective homomorphisms \( g\) and \( g _{F}\) such that the following diagram commutes:
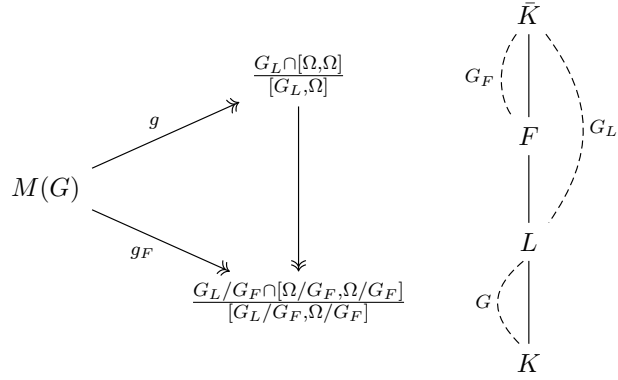
\begin{figure}[htbp]
	\centering
	\begin{tikzcd}[row sep=0.4em, column sep=1.5em, nodes={scale=0.9, transform shape}]
		&&&& {\bar{K}} \\
		&& {\frac{G_L\cap [\Omega, \Omega]}{[G_L, \Omega]}} && \\
		&&&& F \\
		{M(G)} &&&& \\
		&&&& L \\
		&& {\frac{G_L/G_F \cap [\Omega/G_F, \Omega/G_F]}{[G_L/G_F,\Omega/ G_F]}} && \\
		&&&& K
		\arrow["{G_F}"', curve={height=12pt}, dashed, no head, from=1-5, to=3-5]
		\arrow["{G_L}", shift left, curve={height=-18pt}, dashed, no head, from=1-5, to=5-5]
		\arrow[two heads, from=2-3, to=6-3]
		\arrow[no head, from=3-5, to=1-5]
		\arrow["g", two heads, from=4-1, to=2-3]
		\arrow["{g_F}"', two heads, from=4-1, to=6-3]
		\arrow[no head, from=5-5, to=3-5]
		\arrow["G"', shift left, curve={height=12pt}, dashed, no head, from=5-5, to=7-5]
		\arrow[no head, from=7-5, to=5-5]
	\end{tikzcd}
	\caption{Natural Schur multiplicator maps associated with the tower \(K\subset L\subset F\subset\bar K\).}
	\label{fig:schur-multiplicator-maps}
\end{figure}
Moreover, by \cite[Section 6.5, Theorem 4]{Serre_1975}, the homomorphism \(g\) is an isomorphism. If we see \(G\) as the quotient group \( \Omega /G_L\), the homomorphism \(g\) is isomorphic for any extension \( L /K\) which is equivalent to \( M(\Omega) =1\). When \( g _{F}\) is also an isomorphism, we say that the extension \( F\) realizes \( M(G)\).

Actually, for any finite Galois extension of global function fields \(L/K\) with Galois group \(G\), one can choose a finite central extension that realizes \(M(G)\).

\begin{definition}
	Let \(L/K\) be a Galois extension of global function fields and let \(F/K\) be a Galois extension containing \(L\). We say that \(F\) is a central extension of \(L/K\) if \(\operatorname{Gal}(F/L)\) is contained in the center of \(\operatorname{Gal}(F/K)\).
\end{definition}

\begin{proposition}\label{central-existence}
	Let \(L/K\) be a finite Galois extension of global function fields with Galois group \(G\). Then there exists a finite central extension \(F\) of \(L/K\) such that \(F\) realizes \(M(G)\). Furthermore, if \(G\) is an \(\ell\)-group and \(F\) is a finite central extension of \(L/K\) realizing \(M(G)\), then the maximal \(\ell\)-extension of \(K\) contained in \(F\) also realizes \(M(G)\).
\end{proposition}

\begin{proof}
	Let \(K^{\mathrm{ab}}\) denote the maximal abelian extension of \(K\) contained in \(\overline K\). By \cite[Theorem~5]{Miyake_1983}, there exists a finite central extension \(F\) of \(L/K\) such that
	\[
		\operatorname{Gal}\bigl(F/F\cap(K^{\mathrm{ab}}L)\bigr)
		\simeq H^2(G,\mathbb Q/\mathbb Z)^*.
	\]
	Since \(H^2(G,\mathbb Q/\mathbb Z)^*\simeq M(G)\), the two finite groups have the same order.
	On the other hand, since \(F\) is a central extension of \(L/K\), the canonical map
	\[
		g_F:M(G)\twoheadrightarrow
		\operatorname{Gal}\bigl(F/F\cap(K^{\mathrm{ab}}L)\bigr)
	\]
	is surjective.  Hence \(g_F\) is an isomorphism, and \(F\) realizes \(M(G)\).

	The final assertion follows by taking the \(\ell\)-primary part as in \cite[Proposition~3.2]{Frohlich_1983}.
\end{proof}

Even though the above proposition showed that the extension of the Schur multiplicator can be realized by a central extension, it is still not clear how to determine whether a given central extension realizes the Schur multiplicator. But the Pontryagin dual can support a more explicit criterion, which will be used in the proof of the main theorem.

\subsection{Restricted ramification extensions and character groups}

\subsubsection{Restricted ramification }
First, we give some notation for the local and global fields, their unit groups, and idele class groups.
We denoted the completion of \(K\) at \(\frakp\) by \(K_\frakp\), and its unit group by \(U_\frakp:=\mathcal O_\frakp^\times\subset K_\frakp^\times\).

Now, let \(S\) be a finite subset of places of \(K\).
Let
\(
E_{K,S}\defeq
\{a\in K^\times: v _{\mathfrak{p}}(a) =  0 \text{ for every } \mathfrak{p} \notin S\}
\)
be the group of global \(S\)-units.
The \(S\)-idele group of \(K\) is the restricted direct product
\[
	J_K ^{S}	:=\prod_{\frakp\notin S}U_\frakp\prod_{\frakp\in S}K_\frakp^* \triangleq U_f(K)\times K_S^*
\]
of the idele group. Then \(J _{K} ^{S} \cap K ^{\times } = E _{K,S}\). The \(S\)-idele class group of \(K\) is the quotient
\[
	C_{S,K}
	:=\frac{U_f(K)\times K_S^*\times K^\times}{K^\times}
	\simeq
	\frac{U_f(K)\times K_S^*}{E_{K,S}}.
\]
The usual idele class group of \(K\) is
\(
C_K:=J_K/K^\times,
\)
where \(J_K\) is the idele group of \(K\).  The quotient \(C_K/C_{S,K}\) is the \(S\)-ideal class group, denoted by
\(
\Cl_{K,S}.
\)
We have the following classical exact sequence (cf. \cite{Neukirch_2008}):
\begin{equation}\label{eq:S-class-exact}
	1\longrightarrow E_{K,S}
	\longrightarrow U_f(K)\times K_S^*
	\longrightarrow C_K
	\longrightarrow \Cl_{K,S}
	\longrightarrow 1.
\end{equation}

\subsubsection{Character groups}
For a group \(G\), we denote by
\(
X(G)
\)
the torsion subgroup of the Pontryagin dual of \( G\) and \[
	X_\ell(G):=
	\{\phi\in X(G):\phi^{\ell^n}=1\text{ for some }n\geq 0\}
\]
the \(\ell\)-primary part of \(X(G)\).

We are interested in the characters of the idele class group, which can also be viewed as the characters of the idele group trivial on the multiplicative group of the field.

Let \(F\) be a global function field with full constant field \(\F _{q}\), and \(\mathfrak{p}\) be a place of \(F\). Let \(\phi\) be a character of \(C_F\). Its local component \(\phi _{\mathfrak{p}}\) at a prime \(\mathfrak{p}\) of \(F\) is its restriction to \(F _{\mathfrak{p}} ^{\times }\).

\begin{definition}
	A character \(\phi \in X(C_F)\) (or \(X _{\ell}(C _{F})\)) is said to be ramified at \(\mathfrak{p}\) if the local component \(\phi_\frakp\) of \(\phi\) at \(\frakp\) is non-trivial on \(U_\frakp\).  Otherwise, \(\phi\) is unramified at \(\frakp\). Moreover, if \(\phi _{\mathfrak{p}}\) is trivial on \(F _{\mathfrak{p}} ^{\times }\), we say that \(\phi\) is split completely at \(\mathfrak{p}\).
\end{definition}

The torsion subgroup \(X(C_F)\) has an important relationship with the abelian extensions of \(F\).
Specifically, since the completion of \(C_F\) with profinite topology homeomorphisms to the Galois group \(\operatorname{Gal}(F ^{\text{ab}} /F)\), where \(F ^{\text{ab}}\) is the maximal abelian extension of \(F\). The torsion subgroup of the Pontryagin dual of \(C _{F}\), combined with this homeomorphism, gives a bijection between the subgroup \(X(C_F)\) and the set of finite abelian extensions of \(F\). More details can be found in \cite[Chapter 1]{Frohlich_1983}.

Let \(L /K\) be a finite Galois extension of global function fields with Galois group \(\Gamma\).
The Galois group \(\Gamma\) has a canonical action on \(C_L\) (cf. \cite{Lang1994Galois}), which induces an Galois action on \(X(C_L)\), by
\[
	(\gamma\phi)(a)=\phi(\gamma^{-1}a),
	\quad
	\text{where }
	\gamma\in\Gamma,
	\phi\in X(C_L),
	a\in C_L.
\]
We write \(X(C_L)^\Gamma\) and \(X_\ell(C_L)^\Gamma\) to denote the subgroups of \(\Gamma\)-invariant characters.
Let \(M/L\) be a finite abelian extension. By class field theory, \(M  /L\) associates to it the finite character group
\[
	\Phi(M/L) := \{\chi \in X(C_L) : \chi|_{N_{M/L}C_M} = 1\} \cong X(C_L/N_{M/L}C_M),
\]
where
\(
N_{M/L}: C_M \longrightarrow C_L
\)
is the idele class norm map. For the finite Galois extension \(L/K\), we have the following standard exact sequence (cf. \cite[equation 3.11]{Frohlich_1983}):
\begin{equation}\label{eq:cohomological-exact}
	1\longrightarrow X(C_K/N_{L/K}C_L)
	\longrightarrow X(C_K)
	\xrightarrow{\Nhat_{L/K}}
	X(C_L)^\Gamma
	\xrightarrow{r}
	X\bigl(H^{-1}(\Gamma,C_L)\bigr)
	\longrightarrow 1,
\end{equation}
where \(\Nhat_{L/K}(\psi)=\psi\circ N_{L/K}\). For \(\phi\in X(C_L)^\Gamma\), the restriction of \(\phi\) to \(\Ker(N_{L/K})\) is trivial on the subgroup generated by the elements \(\gamma a/a\), with \(\gamma\in\Gamma\) and \(a\in C_L\), and hence induces a character of
\[
	H^{-1}(\Gamma,C_L)
	=\Ker(N_{L/K})/\langle \gamma a/a:\gamma\in\Gamma,\ a\in C_L\rangle.
\]
We denote this induced character by \(r(\phi)\). In particular, exactness gives
\[
	\operatorname{Im}(\Nhat_{L/K})=\Ker(r).
\]

We keep the above notation. The following proposition gives a criterion for an abelian extension to be central over the base field \(K\), and whether a central extension realises the Schur multiplicator.
\begin{proposition}\cite[Proposition 3.4]{Frohlich_1983}\label{central-criterion}
	Let $M/K$ be a Galois extension. Then:
	\begin{enumerate}
		\item $M$ is a central extension of $L/K$ if and only if $\Phi(M/L) \subseteq X(C_L)^\Gamma$.
		\item $M$ realizes $M(\Gamma)$ if and only if $r\bigl(\Phi(M/L)\bigr) = X\bigl(H^{-1}(\Gamma,C_L)\bigr)$.
	\end{enumerate}
\end{proposition}

\subsection{A finite subgroup replacement lemma}

The following elementary observation is used to replace an arbitrary finite character group by one satisfying the desired ramification restrictions.

\begin{lemma}\label{lem:finite-replacement}\cite[Lemma 2.9]{Watt_1985}
	Let \(G\) be a torsion abelian group and suppose that
	\(
	G=G_1G_2.
	\)
	For a group \(\overline{G}\), let \(\rho:G\to \overline G\) be a homomorphism with \(G_2\subset\Ker(\rho)\).  If \(H\subset G\) is a finite subgroup such that \(\rho(H)=\overline G\), then there exists a finite subgroup
	\(
	H_1\subset G_1
	\)
	such that
	\(
	\rho(H_1)=\overline G.
	\)
\end{lemma}

We shall use the standard extension theorem for characters of locally compact abelian groups; see \cite[Corollary of Theorem~27]{Morris_1977}.

\begin{lemma}[Character extension]\label{lem:character-extension}
	Let \(A\) be a locally compact abelian group and let \(B\subset A\) be a closed subgroup. Then every continuous character of \(B\) extends to a continuous character of \(A\).
\end{lemma}

\section{Local descent of invariant characters}\label{sec:local}

We first prove the local statement used to construct the auxiliary global character on \(K\).

\begin{proposition}[Local norm descent]
	\label{prop:local-norm-descent}

	Let \(L/K\) be a finite Galois extension of global function fields with Galois group \(\Gamma\), and let
	\(\mathfrak p\) be a place of \(K\) which is unramified in \(L/K\). Then the homomorphism
	\[
		\begin{aligned}
			X_{\ell}\bigl(K_{\mathfrak p}^{\times}\bigr)
			 & \longrightarrow
			X_{\ell}\left(
			\prod_{\mathfrak P\mid\mathfrak p}
			L_{\mathfrak P}^{\times}
			\right)^{\Gamma},  \\
			\chi_{\mathfrak p}
			 & \longmapsto
			\left(
			\chi_{\mathfrak p}\circ
			N_{L_{\mathfrak P}/K_{\mathfrak p}}
			\right)_{\mathfrak P\mid\mathfrak p}
		\end{aligned}
	\]
	is surjective.
	Moreover, if
	\(
	\phi\in
	X_{\ell}\left(
	\prod_{\mathfrak P\mid\mathfrak p}
	L_{\mathfrak P}^{\times}
	\right)^{\Gamma}
	\)
	is unramified at every place \(\mathfrak P\mid\mathfrak p\), then one may choose
	\(\chi_{\mathfrak p}\in X_{\ell}(K_{\mathfrak p}^{\times})\)
	to be unramified as well.

\end{proposition}
\begin{proof}
	Fix a place \(\mathfrak P_0\mid \mathfrak p\), and let
	\[
		F:=K_{\mathfrak p},\qquad
		E:=L_{\mathfrak P_0},\qquad
		D:=D_{\mathfrak P_0},
	\]
	where \(D\) is the decomposition group of \(\mathfrak P_0\) in \(L/K\).  	Restriction to the \(\mathfrak P_0\)-component gives an isomorphism
	\[
		X_{\ell}\left(
		\prod_{\mathfrak P\mid \mathfrak p}L_{\mathfrak P}^{\times}
		\right)^{\Gamma}
		\;\xrightarrow{\;\sim\;}\;
		X_{\ell}(E^{\times})^{D}.
		\tag{a}
	\]
	Indeed, whether a \(\Gamma\)-invariant character is trivial determined by its \(\mathfrak{P}_0\)-component, because \(\Gamma\) acts transitively on the set of primes \(\mathfrak{P}\) above \(\mathfrak{p}\). Conversely, if \(\psi\in X_{\ell}(E^\times)^D\), then for
	\(\mathfrak P=\gamma(\mathfrak P_0)\) one defines
	\[
		\phi_{\mathfrak P}(x):=\psi(\gamma^{-1}x),
		\qquad x\in L_{\mathfrak P}^{\times},
	\]
	and this is well-defined precisely because \(\psi\) is \(D\)-invariant.

	Since \(\mathfrak p\) is unramified in \(L/K\), the local extension
	\(E/F\) is unramified, hence \(D=\operatorname{Gal}(E/F)\) is cyclic.
	Let
	\[
		I_D E^\times
		:=
		\left\langle \delta a/a : \delta\in D,\ a\in E^\times \right\rangle .
	\]
	By Hilbert 90 \cite[Corollary in p.151]{Serre_1979},
	\(
	\ker N_{E/F}=I_D E^\times.
	\)
	Thus we have an exact sequence.
	\[
		1
		\longrightarrow
		I_D E^\times
		\longrightarrow
		E^\times
		\xrightarrow{\,N_{E/F}\,}
		F^\times
		\longrightarrow
		F^\times/N_{E/F}E^\times
		\longrightarrow
		1.
		\tag{b}
	\]
	A character of \(E^\times\) is \(D\)-invariant exactly when it is
	trivial on \(I_D E^\times\). Hence every
	\(\psi\in X_\ell(E^\times)^D\) descends to a character
	\[
		\overline\psi:N_{E/F}E^\times\longrightarrow\C^\times,
		\qquad
		\overline\psi(N_{E/F}a)=\psi(a).
	\]
	Since \(E/F\) is unramified, \(N_{E/F}E^\times\) is an open subgroup of finite index in \(F^\times\). By Lemma~\ref{lem:character-extension}, \(\overline\psi\) extends to a continuous character of \(F^\times\). As the quotient \(F^\times/N_{E/F}E^\times\) is finite, the extension has finite order. Its prime-to-\(\ell\) component is trivial on \(N_{E/F}E^\times\), so its \(\ell\)-primary component has the same restriction as \(\overline\psi\). Thus we obtain \(\chi\in X_\ell(F^\times)\) satisfying \(\psi=\chi\circ N_{E/F}\). Consequently, we have the exact character sequence
	\[
		0
		\longrightarrow
		X_{\ell}\bigl(F^\times/N_{E/F}E^\times\bigr)
		\longrightarrow
		X_{\ell}(F^\times)
		\xrightarrow{\,\Nhat_{E/F}\,}
		X_{\ell}(E^\times)^D
		\longrightarrow
		0,
		\tag{c}
	\]
	Where \(\Nhat_{E/F}(\chi)=\chi\circ N_{E/F}\).

	Under the identification (a), the map in the statement is exactly the
	map \(\Nhat_{E/F}\) in (c). Hence it is surjective.

	Now suppose that
	\(
	\phi\in
	X_{\ell}\left(
	\prod_{\mathfrak P\mid \mathfrak p}L_{\mathfrak P}^{\times}
	\right)^{\Gamma}
	\)
	is unramified at every \(\mathfrak P\mid\mathfrak p\). Let
	\(\psi=\phi_{\mathfrak P_0}\in X_{\ell}(E^\times)^D\). Then \(\psi\)
	is trivial on \(\mathcal O_E^\times\). Choose
	\(\chi\in X_{\ell}(F^\times)\) such that
	\(
	\psi=\chi\circ N_{E/F}.
	\)
	Since \(E/F\) is unramified, the norm on units is surjective.
	Thus for every \(u\in \mathcal O_F^\times\), we may write
	\(u=N_{E/F}(v)\) with \(v\in \mathcal O_E^\times\), and then
	\[
		\chi(u)=\chi(N_{E/F}v)=\psi(v),
	\]
	which is trivial. Hence \(\chi\) is trivial on \(\mathcal O_F^\times\),
	i.e. \(\chi\) is unramified.
\end{proof}

\section{The global character factorisation}\label{sec:factorisation}

We now prove the main character-theoretic statement, which will allow us to choose a special central extension realizing the Schur multiplicator in the final step.

Let $T$ be a finite set of finite primes of $K$. Define
\[
	X_\ell(C_L)_{T,S}
	=
	\left\{
	\phi \in X_\ell(C_L) :
	\begin{array}{l}
		\phi \text{ is unramified for every } \mathfrak{P} \mid \mathfrak{p} \text{ with } \mathfrak{p} \notin T \cup S, \\[2pt]
			ext{and is split completely for every } \mathfrak{P}' \mid \mathfrak{p}' \text{ with } \mathfrak{p}' \in S
	\end{array}
	\right\}.
\]
and
\[
	X_\ell(C_K)_{T}
	:=\{\psi\in X_\ell(C_K):
	\psi\text{ is unramified for every }\mathfrak p\in T  \}.
\]

\begin{theorem}[Character factorisation]\label{thm:character-factorisation}
	Let \(K\) be a global function field with full constant field \(\F_q\). Let \(S\) be a non-empty finite set of places of \(K\), let \(T\) be a finite set of finite places disjoint from \(S\), and let \(L/K\) be a finite Galois extension with group \(\Gamma\), unramified outside \(T\), in which every place of \(S\) splits completely. Then the following hold.
	\begin{enumerate}
		\item If \(S=\{\infty\}\) and \(\ell\nmid q-1\), then
		\begin{equation}\label{eq:main-factorisation}
			X_\ell(C_L)^\Gamma
			=X_\ell(C_L)_{T,S}^\Gamma\cdot \Nhat_{L/K}\bigl(X_\ell(C_K)_T\bigr).
		\end{equation}
		\item If \(T\ne\varnothing\) and either \(S=\{\infty\}\) or \(\ell=p=\operatorname{char}K\), then
		\begin{equation}\label{eq:main-factorisation-II}
			X_\ell(C_L)^\Gamma
			=X_\ell(C_L)_{T,S}^\Gamma\cdot \Nhat_{L/K}\bigl(X_\ell(C_K)\bigr).
		\end{equation}
	\end{enumerate}
\end{theorem}

\begin{proof}[Proof of Theorem~\ref{thm:character-factorisation}, Case~\textup{(1)}]
Since both \(X_\ell(C_L)_{T,S}^\Gamma\) and \(\Nhat_{L/K}(X_\ell(C_K)_T)\) are subgroups of \(X_\ell(C_L)^\Gamma\), only the reverse inclusion requires proof. Let \(\phi\in X_\ell(C_L)^\Gamma\), and put
	\[
		R:=\{\frakp:\frakp\notin T\cup S\text{ and }\phi_\frakP\text{ is ramified for some }\frakP\mid\frakp\}.
	\]
	The set \(R\) is finite.
	\begin{figure}[ht]
			\centering
			\begin{tikzpicture}[
					scale=0.4,
					every node/.style={font=\Large},
					region/.style={draw=black, line width=0.8pt},
					outer/.style={draw=black, line width=0.9pt}
				]
	
				\draw[outer] (0,0) ellipse (4.6cm and 3.0cm);
	
				\draw[region] (-1.7,1.0) circle (1.35cm);
				\node at (-1.7,1.0) {$T$};
	
				\draw[region] (1.8,0.65) circle (0.9cm);
				\node at (1.8,0.65) {$S$};
	
				\draw[region] (-0.35,-1.35) circle (1.1cm);
				\node at (-0.35,-1.35) {$R$};
	
			\end{tikzpicture}
			\caption{Relation among the prescribed ramification set \(T\), the splitting set \(S\), and the correction set \(R\).}
			\label{fig:sets-tsr}
		\end{figure}
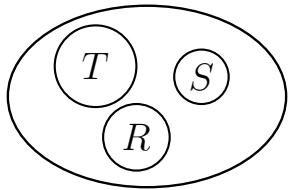
	For each \(\frakp\in R\), Proposition~\ref{prop:local-norm-descent} gives \(\chi_\frakp\in X_\ell(K_\frakp^\times)\) such that \(\phi_\frakP=\chi_\frakp\circ N_{L_\frakP/K_\frakp}\). Since \(S=\{\infty\}\) splits completely, we identify the local component above \(\infty\) with a character \(\phi_\infty\in X_\ell(K_\infty^\times)\). Define
	\[
			heta:U_f(K)\times K_\infty^\times\longrightarrow\C^\times,
		\qquad
			heta((u_\frakp)_\frakp)
		=\left(\prod_{\frakp\in R}\chi_\frakp(u_\frakp)\right)\phi_\infty(u_\infty).
	\]
	For \(x\in E_{K,S}\), every factor of \(\theta(x)\) has \(\ell\)-power order. Since \(S=\{\infty\}\), the degree argument gives \(E_{K,S}=\F_q^\times\), so \(\theta(x)\) also has order dividing \(q-1\). As \(\ell\nmid q-1\), we have \(\theta(x)=1\). Thus \(\theta\) descends to a finite \(\ell\)-primary character of \(C_{S,K}\), and Lemma~\ref{lem:character-extension} gives \(\psi\in X_\ell(C_K)\) with the same restriction to \(U_f(K)\times K_\infty^\times\). Since \(\theta\) is trivial on \(U_\frakp\) for \(\frakp\in T\), we have \(\psi\in X_\ell(C_K)_T\). The local definition of \(\theta\) then shows that \(\phi\Nhat_{L/K}(\psi^{-1})\) is unramified outside \(T\) and trivial above \(S\). Hence it lies in \(X_\ell(C_L)_{T,S}^\Gamma\), proving \eqref{eq:main-factorisation}.
\end{proof}

For the second part we use a single pro-\(\ell\) embedding statement. Let \(G\) be a topological group. We denote by
\[
	\widetilde{G}^\ell = \varprojlim_{N} G/N
\]
with \(N\) running through the open normal subgroups of \(G\) such that \(G/N\) is an \(\ell\)-group, the pro-\(\ell\) completion of \(G\). Then there exists a natural continuous homomorphism \(G\to \widetilde{G}^\ell\) with dense image. More details about the pro-\(\ell\) completion can be found in \cite[Chapter 1]{Frohlich_1983} or \cite{Neukirch_2008}.

\begin{lemma}[Pro-\(\ell\) embedding of \(U\)-units]\label{lem:leopoldt-embedding}
	Let \(K\) be a global function field with full constant field \(\F_q\), let \(S\) be a non-empty finite set of places of \(K\), and let \(T\) be a non-empty finite set of finite places disjoint from \(S\). Put \(U=S\cup T\). Assume either
	\begin{enumerate}
		\item \(S=\{\infty\}\) and \(\ell\nmid q\), or
		\item \(\ell=p=\operatorname{char}K\).
	\end{enumerate}
	Then the diagonal embedding \(E_{K,U}\to\prod_{t\in T}K_t^\times\) induces a continuous injection
	\[
		\widetilde{E_{K,U}}^\ell
		\longrightarrow
		\prod_{t\in T}\widetilde{K_t^\times}^\ell.
	\]
	In case~\textup{(2)}, for every \(t\in T\) the map \(\widetilde{E_{K,U}}^p\to\widetilde{K_t^\times}^p\) is already injective.
\end{lemma}

\begin{proof}
	By \cite[Proposition 14.2]{Rosen_2002}, \(E_{K,U}\) is a finitely generated abelian group.

	Assume first that \(S=\{\infty\}\) and \(\ell\nmid q\). The divisor map gives a left-exact sequence
	\[
		1\longrightarrow \F_q^\times
		\longrightarrow E_{K,U}
		\xrightarrow{\operatorname{div}}
		\Div_U^0(K),
	\]
	where \(\Div_U^0(K)\) denotes the group of degree-zero divisors supported on \(U\). For a finitely generated abelian group \(A\), its pro-\(\ell\) completion is canonically identified with \(A\otimes_\Z\Z_\ell\). Since \(\Z_\ell\) is flat over \(\Z\), we obtain the left exact sequence
	\begin{equation}\label{eq:completed-unit-divisor}
		1\longrightarrow
		\F_q^\times\otimes_\Z\Z_\ell
		\longrightarrow
		\widetilde{E_{K,U}}^{\ell}
		\xrightarrow{\widetilde{\operatorname{div}}}
		\widetilde{\Div_U^0(K)}^{\ell}.
	\end{equation}
	For every \(t\in T\), the valuation \(v_t:K_t^\times\to\Z\) induces a continuous homomorphism \(\widetilde v_t:\widetilde{K_t^\times}^{\ell}\to\Z_\ell\).
	The divisor and valuation maps are compatible with pro-\(\ell\) completion, as shown in the following commutative diagram:
	\begin{figure}[htbp]
			\centering
			\begin{tikzcd}[column sep=2.25em]
				{E _{K,U}} && {\mathrm{Div} _{U} ^{0}(K)} & {\widetilde{E _{K,U}} ^{\ell}} && {\widetilde{\mathrm{Div} _{U} ^{0}(K)} ^{\ell}} \\
				\\
				{K_t^{\times}} && {\mathbb{Z}} & {\widetilde{K_t^{\times}} ^{\ell}} && {\mathbb{Z}_{\ell}}
				\arrow["{\mathrm{div}}", from=1-1, to=1-3]
				\arrow[hook, from=1-1, to=3-1]
				\arrow[""{name=0, anchor=center, inner sep=0}, "{m_t}"', from=1-3, to=3-3]
				\arrow["{\widetilde{\mathrm{div}}}", from=1-4, to=1-6]
				\arrow[""{name=1, anchor=center, inner sep=0}, from=1-4, to=3-4]
				\arrow[from=1-6, to=3-6]
				\arrow["{v_t}", from=3-1, to=3-3]
				\arrow[from=3-4, to=3-6]
				\arrow["{\text{induce}}"{pos=0.7}, curve={height=6pt}, between={0.1}{0.9}, dashed, from=0, to=1]
			\end{tikzcd}
			\caption{Compatibility of divisor and valuation maps after pro-\(\ell\) completion.}
			\label{fig:pro-ell-divisor-compatibility}
		\end{figure}
	Here \(m_t\) denotes the map taking a divisor supported on \(U\) to its coefficient at \(t\), and \(v_t\) is the valuation at \(t\).

	Let \(x\in\widetilde{E_{K,U}}^{\ell}\) have trivial image in \(\prod_{t\in T}\widetilde{K_t^\times}^{\ell}\). Write
	\[
		\widetilde{\operatorname{div}}(x)
		=a_\infty\,\infty+\sum_{t\in T}a_t t,
		\qquad a_\infty,a_t\in\Z_\ell.
	\]
	By the commutativity of Figure~\ref{fig:pro-ell-divisor-compatibility}, the image of \(\widetilde{\operatorname{div}}(x)\) in \(\Z_\ell\) is trivial for every \(t\in T\). This image is precisely the coefficient \(a_t\). Hence \(a_t=0\) for every \(t\in T\). Since the divisor has degree zero, \(a_\infty\deg(\infty)=0\) in \(\Z_\ell\), and hence \(a_\infty=0\). Thus \(\widetilde{\operatorname{div}}(x)=0\), so \eqref{eq:completed-unit-divisor} gives \(x\in\F_q^\times\otimes_\Z\Z_\ell\). For any \(t\in T\), the \(\ell\)-primary subgroup of \(\F_q^\times\) embeds into the \(\ell\)-primary subgroup of the residue field at \(t\). Since \(\ell\nmid q\), the kernel of the reduction map \(\mathcal O_t^\times\to\F_{q^{\deg(t)}}^\times\) is pro-\(p\) and contains no non-trivial element of \(\ell\)-power order. Hence \(\F_q^\times\otimes_\Z\Z_\ell\to\widetilde{K_t^\times}^{\ell}\) is injective, and therefore \(x=1\).

	Now assume \(\ell=p=\operatorname{char}K\), and fix \(t\in T\). Since \(|\F_q^\times|=q-1\) is prime to \(p\), \(\widetilde{E_{K,U}}^p\) is a finite free \(\Z_p\)-module. Moreover,
	\[
		K_t^\times\simeq \langle\pi_t\rangle\times \F_{q^{\deg(t)}}^\times\times U_t^{(1)},
	\]
	where \(U_t^{(1)}=1+\mathfrak m_t\). Thus \(\widetilde{K_t^\times}^{p}\simeq \Z_p\times U_t^{(1)}\), which has no non-trivial \(p\)-torsion, and the natural map
	\[
		K_t^\times/(K_t^\times)^p
		\longrightarrow
		\widetilde{K_t^\times}^{p}/(\widetilde{K_t^\times}^{p})^p
	\]
	is an isomorphism.

	We claim that
	\begin{equation}\label{eq:p-pure-units}
		E_{K,U}\cap (K_t^\times)^p=E_{K,U}^p.
	\end{equation}
	Indeed, let \(a\in E_{K,U}\) and suppose that \(a=b^p\) in \(K_t\). Choose \(\pi_t\in K\) with \(v_t(\pi_t)=1\). Then \(d\pi_t\ne0\), and the differential criterion over the perfect constant field gives
	\[
		da=0\quad\Longleftrightarrow\quad a\in K^p;
	\]
	see \cite[Chapter~4]{Stichtenoth_2009}. Writing \(da=f\,d\pi_t\) with \(f\in K\), its image in the completed differential module is zero because \(a\) is a \(p\)-th power in \(K_t\). Since \(d\pi_t\ne0\), this forces \(f=0\), so \(da=0\), and therefore \(a=c^p\) for some \(c\in K\). For every place \(v\notin U\),
	\[
		0=v(a)=p\,v(c),
	\]
	so \(c\in E_{K,U}\). This proves \eqref{eq:p-pure-units}.

	Consequently
	\[
		E_{K,U}/E_{K,U}^p
		\longrightarrow
		\widetilde{K_t^\times}^{p}/(\widetilde{K_t^\times}^{p})^p
	\]
	is injective. If \(x\) lies in the kernel of \(\widetilde{E_{K,U}}^p\to\widetilde{K_t^\times}^p\) and \(x\ne0\), write \(x=p^m y\) with \(y\notin p\widetilde{E_{K,U}}^p\). Since \(\widetilde{K_t^\times}^p\) has no non-trivial \(p\)-torsion, the image of \(y\) is trivial. Reducing modulo \(p\) contradicts the preceding injectivity. Hence \(x=0\). This proves the assertion for every \(t\in T\), and therefore also for the diagonal map. Continuity in both cases follows from functoriality of pro-\(\ell\) completion; cf. \cite[Lemma 3.2.3]{ProfiniteGroups}.
\end{proof}

\begin{proof}[Proof of Theorem~\ref{thm:character-factorisation}, Case~\textup{(2)}]
Put \(U=S\cup T\). Let \(\phi\in X_\ell(C_L)^\Gamma\) and define
	\[
		R:=\{\frakp\notin T\cup S:\phi_\frakP\text{ is ramified for some }\frakP\mid\frakp\}.
	\]
	For every \(\frakp\in R\), Proposition~\ref{prop:local-norm-descent} gives \(\chi_\frakp\in X_\ell(K_\frakp^\times)\) whose local norm pullback is \(\phi_\frakP\). For each \(s\in S\), complete splitting identifies \(L_v\) with \(K_s\) for every \(v\mid s\); by \(\Gamma\)-invariance the corresponding local character is independent of \(v\), and we denote it by \(\phi_s\). Define
	\[
			heta:E_{K,U}\longrightarrow\C^\times,
		\qquad
			heta(a)=\prod_{\frakp\in R}\chi_\frakp(a)\prod_{s\in S}\phi_s(a).
	\]
	Its image has finite \(\ell\)-power order, so \(\theta\) extends continuously to \(\theta':\widetilde{E_{K,U}}^\ell\to\C^\times\). Under either hypothesis in part~\textup{(2)}, Lemma~\ref{lem:leopoldt-embedding} identifies \(\widetilde{E_{K,U}}^\ell\) with a closed subgroup of
	\[
		P:=\prod_{t\in T}\widetilde{K_t^\times}^{\ell}.
	\]
	By Lemma~\ref{lem:character-extension}, \(\theta'\) extends to a continuous character \(\rho:P\to\C^\times\). Since \(P\) is pro-\(\ell\), \(\rho\) has finite \(\ell\)-power order. Let \(\rho'\) be its pullback to \(\prod_{t\in T}K_t^\times\).
	\begin{figure}[htbp]
			\centering
			\begin{tikzcd}
				{\prod_{t\in T}K_t^{\times}} & {\prod_{t\in T}\widetilde{K_t^{\times}} ^{\ell}} & {\mathbb{C}^{\times}} \\
				& {\widetilde{E _{K,U}} ^{\ell}}
				\arrow["\tau", from=1-1, to=1-2]
				\arrow["\rho", from=1-2, to=1-3]
				\arrow[from=2-2, to=1-2]
				\arrow["{\theta'}"'{pos=0.3}, from=2-2, to=1-3]
			\end{tikzcd}
			\caption{Extension of the pro-\(\ell\) character \(\theta'\) to the local components above \(T\).}
			\label{fig:theta-extension}
		\end{figure}
	We define
	\[
		\psi'((x_\frakp)_\frakp)
		=(\rho')^{-1}((x_t)_{t\in T})
		\prod_{\frakp\in R}\chi_\frakp(x_\frakp)
		\prod_{s\in S}\phi_s(x_s)
	\]
	on \(J_K^U\), taking it to be trivial on all remaining unit components. By construction \(\psi'(a)=1\) for every \(a\in E_{K,U}\). Hence \(\psi'\) descends to a finite \(\ell\)-primary character of \(C_{U,K}\). Since \(C_{U,K}\) has finite index in \(C_K\), Lemma~\ref{lem:character-extension} gives \(\psi\in X_\ell(C_K)\) with \(\psi|_{J_K^U}=\psi'\).

	Set \(\Phi=\phi\Nhat_{L/K}(\psi^{-1})\). For \(\frakp\notin T\cup S\), the local norm construction shows that \(\Phi\) is unramified above \(\frakp\). If \(s\in S\) and \(v\mid s\), then \(L_v=K_s\), the local norm is the identity, and \(\psi_s=\phi_s\); hence \(\Phi_v=1\). Therefore \(\Phi\in X_\ell(C_L)_{T,S}^\Gamma\), proving \eqref{eq:main-factorisation-II}.
\end{proof}

\section{The multiplicator-free theorem}\label{sec:main}

We now apply the character factorisation theorem to the Schur multiplicator.

\begin{theorem}[Restricted ramification and multiplicator freeness]\label{thm:main}
Let \(K\) be a global function field with full constant field
\(\mathbb F_q\), and let \(T\) be a finite set of finite places of \(K\).
Then the following hold.
\begin{enumerate}
\item If \(S=\{\infty\}\) and either \(\ell\nmid(q-1)\) or \(T\neq\varnothing\), then
\(\operatorname{Gal}(K(\ell,T,S)/K)\) is multiplicator free.
\item If \(p=\operatorname{char}K\), \(\ell=p\), and \(T\neq\varnothing\),
then the same conclusion holds for every non-empty finite set \(S\) disjoint from \(T\).
\end{enumerate}
In both cases, for every finite Galois subextension
\[
L/K\subseteq K(\ell,T,S)/K,
\]
the Schur multiplier of \(\operatorname{Gal}(L/K)\) admits a central
realization inside \(K(\ell,T,S)\).
\end{theorem}

\begin{proof}
	Fix an algebraic closure \(K^{\ac}\) of \(K\), and put
	\[
		\Omega:=\Gal(K^{\ac}/K),
		\qquad
		A:=\Gal(K^{\ac}/K(\ell,T,S)).
	\]
	Then \(\Omega/A\simeq \Gal(K(\ell,T,S)/K)\).
	We must prove that \(\Omega/A\) is multiplicator free.  Let \(B/A\) be an arbitrary open normal subgroup of \(\Omega/A\), with \(B\triangleleft\Omega\) the corresponding open normal subgroup containing \(A\).  Put
	\[
		L:=(K^{\ac})^B,
		\qquad
		\Gamma:=\Gal(L/K)\simeq \Omega/B.
	\]
	Since \(A\subset B\), we have \(L\subset K(\ell,T,S)\).
	Thus \(L/K\) is a finite \(\ell\)-extension, unramified at every finite place outside \(T\), and every place in \(S\) splits completely in \(L/K\).
	By the definition of multiplicator-free, it suffices to show that the natural map
	\begin{equation}\label{eq:gA}
		g_A:M(\Gamma)
		\longrightarrow
		\frac{B/A\cap(\Omega/A,\Omega/A)}{(B/A,\Omega/A)}
	\end{equation}
	is an isomorphism.

	If there exists a finite central extension \(F\) of \(L /K\) contained in \(K(\ell, T, S)\) such that it realises the Schur multiplicator \(M(\Gamma)\), we are done.
	By Proposition \ref{central-existence}, there exists a finite central \(\ell\)-extension \(M/L\), with \(L\subset M\subset\overline K\), which is Galois over \(K\) and realises \(M(\Gamma)\). By Proposition \ref{central-criterion}, we have
	\[
		\Phi(M/L)\subset X_\ell(C_L)^\Gamma \text{ and }
		r\bigl(\Phi(M/L)\bigr)=X\bigl(H^{-1}(\Gamma,C_L)\bigr).
	\]
	In case~(1), if \(\ell\nmid q-1\), apply Theorem~\ref{thm:character-factorisation}\textup{(1)} and put \(Y=X_\ell(C_K)_T\). If \(\ell\mid q-1\), then \(T\ne\varnothing\) and Theorem~\ref{thm:character-factorisation}\textup{(2)} applies; in this case put \(Y=X_\ell(C_K)\). In case~(2), Theorem~\ref{thm:character-factorisation}\textup{(2)} applies with \(\ell=p\), and we put \(Y=X_p(C_K)\). In either case,
	\[
		X_\ell(C_L)^\Gamma
		=X_\ell(C_L)_{T,S}^\Gamma\cdot \Nhat_{L/K}(Y),
	\]
	and by exact sequence \eqref{eq:cohomological-exact}, \(\Nhat_{L/K}(Y)\subset \Ker(r)\).
	Now, using Lemma \ref{lem:finite-replacement} with
	\[
		G=X_\ell(C_L)^\Gamma,
		\quad
		G_1=X_\ell(C_L)_{T,S}^\Gamma,
		\quad
		G_2=\Nhat_{L/K}(Y),
	\]
	we obtain a finite subgroup
	\(
	H_1\subset X_\ell(C_L)_{T,S}^\Gamma
	\)
	such that
	\begin{equation*}
		r(H_1)=X\bigl(H^{-1}(\Gamma,C_L)\bigr).
	\end{equation*}

	By class field theory, the finite subgroup \(H_1\subset X _{\ell}(C_L)\) corresponds to a finite abelian \(\ell\)-extension \(F/L\) such that \(\Phi(F/L)=H_1\). Set \(N:=\bigcap_{\chi\in H_1}\Ker(\chi)=N_{F/L}C_F\).
	Since \(H_1\subset X_\ell(C_L)^\Gamma\), for every \(\gamma\in\Gamma\), \(\chi\in H_1\), and \(a\in N\),
	\[
		\chi(\gamma a)=(\gamma^{-1}\chi)(a)=\chi(a)=1.
	\]
	Hence \(N\) is \(\Gamma\)-stable. By the Galois equivariance of the class field correspondence, \(F/K\) is Galois. Since \(\Phi(F/L)=H_1\subset X(C_L)^\Gamma\), Proposition~\ref{central-criterion} shows that \(F\) is a central extension of \(L/K\).

	Moreover, every character in \(H_1\subset X_\ell(C_L)_{T,S}\) is unramified at the places of \(L\) above finite places outside \(T\), and is trivial on \(L_v^\times\) for every \(v\mid S\). By local class field theory, \(F/L\) is therefore unramified outside \(T\), and every place above \(S\) splits completely in \(F/L\). Since \(L/K\) has the same local properties, so does \(F/K\). Finally, \(\Gal(F/L)\) and \(\Gal(L/K)\) are both \(\ell\)-groups, hence \(\Gal(F/K)\) is an \(\ell\)-group. Therefore \(F\subset K(\ell,T,S)\).
	Moreover,
	\[
		r\bigl(\Phi(M/L)\bigr)
		=r\bigl(\Phi(F/L)\bigr)
		=X\bigl(H^{-1}(\Gamma,C_L)\bigr).
	\]
	By Proposition \ref{central-criterion}, the central extension \(F/L\) also realises the Schur multiplicator \(M(\Gamma)\).  Let \(H _{F}: = \operatorname{Gal}(\overline{K} /F)\). Hence the natural map
	\begin{equation}\label{eq:gF}
		g_F:M(\Gamma)
		\longrightarrow
		\frac{B/H_F\cap(\Omega/H_F,\Omega/H_F)}{(B/H_F,\Omega/H_F)}
	\end{equation}
	is an isomorphism.
	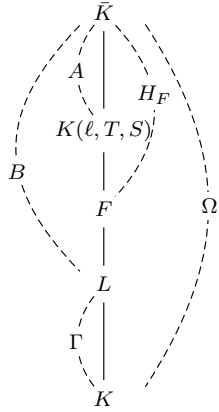
\begin{figure}[H]
		\centering
		\begin{tikzcd}[
				row sep=scriptsize,
				cells={nodes={font=\scriptsize}},
				every label/.append style={font=\scriptsize},
				dashedline/.style={
						dashed,
						no head,
						shorten <= -3pt,
						shorten >= -3pt
					}
			]
			{\bar{K}} \\
			\\
			{K(\ell,T,S)} \\
			F \\
			L \\
			\\
			K
			\arrow[no head, from=1-1, to=3-1]
			\arrow["A"{description}, curve={height=12pt}, dashedline, from=1-1, to=3-1]
			\arrow["{H_F}"{description, pos=0.4}, curve={height=-24pt}, dashedline, from=1-1, to=4-1]
			\arrow["B"{description, pos=0.6}, shift right=2, curve={height=30pt}, dashedline, from=1-1, to=5-1]
			\arrow["\Omega"{description}, shift left=5, curve={height=-30pt}, dashedline, from=1-1, to=7-1]
			\arrow[no head, from=3-1, to=4-1]
			\arrow[no head, from=4-1, to=5-1]
			\arrow[no head, from=5-1, to=7-1]
			\arrow["\Gamma"{description}, curve={height=12pt}, dashedline, from=5-1, to=7-1]
		\end{tikzcd}
		\caption{Field tower used to compare the maps \(g_A\) and \(g_F\).}
		\label{fig:field-tower-comparison}
	\end{figure}
	The quotient map \(\Omega/A\twoheadrightarrow\Omega/H_F\) induces a natural map from the target of \(g_A\) to the target of \(g_F\), and the resulting triangle commutes:
	\begin{figure}[htbp]
		\centering
		\begin{tikzcd}[column sep=large]
			M(\Gamma) \arrow[r,"g_A"] \arrow[dr,"g_F"']
			& \dfrac{B/A\cap(\Omega/A,\Omega/A)}{(B/A,\Omega/A)} \arrow[d]
			\\
			& \dfrac{B/H_F\cap(\Omega/H_F,\Omega/H_F)}{(B/H_F,\Omega/H_F)} .
		\end{tikzcd}
		\caption{Commutative comparison of the natural maps \(g_A\) and \(g_F\).}
		\label{fig:g-map-comparison}
	\end{figure}
	Since \(g_F\) is an isomorphism and \(g_A\) is always a natural epimorphism, the map \(g_A\) must be injective as well: if \(g_A(x)=0\), then \(g_F(x)=0\), and hence \(x=0\).  Therefore \(g_A\) is an isomorphism.

	The subgroup \(B/A\) was arbitrary, so \(\Omega/A\simeq\Gal(K(\ell,T,S)/K)\) is multiplicator free.
\end{proof}

\bibliographystyle{plain}
\bibliography{multiplicator_freeness_unified_results_refs}

\end{document}